\documentclass{amsart}
\usepackage[english]{babel}
\usepackage[T1]{fontenc}
\usepackage[cp1250]{inputenc}
\usepackage{amsmath}
\usepackage{amsfonts}
\usepackage{amsthm}
\usepackage{graphicx}
\usepackage{caption}
\usepackage{subcaption}

\newtheorem{theorem}{Theorem}
\newtheorem{lemma}{Lemma}
\newtheorem{conj}{Problem}
\newtheorem{defini}{Definition}

\title[Li-Yorke chaos and positive topological sequence entropy in NDS]{Relationship between Li-Yorke chaos and positive topological sequence entropy in nonautonomous dynamical systems}
\author[J. \v{S}otola]{Jakub \v{S}otola}
\address{Jakub \v{S}otola, Mathematical Institute, Silesian University, Na Rybn\'{i}\v{c}ku 1, 746 01, Opava, Czech Republic}
\email{Jakub.Sotola@math.slu.cz}

\begin{document}
\maketitle

\section*{Abstract}
We study chaotic properties of uniformly convergent nonautonomous dynamical systems. We show that, contrary to the autonomous systems on the compact interval, positivity of topological sequence entropy and occurrence of Li-Yorke chaos are not equivalent, more precisely, neither of the two possible implications is true.

\section{Introduction}
Among other notions, Li-Yorke chaos and topological entropy belong to basic and widely used concepts in the theory of discrete dynamical systems. The question of their mutual relationship is thus very natural. Since 2002, it is known (see \cite{BGKM}) that for continuous mappings of compact metric spaces positive topological entropy implies Li-Yorke chaos. Analogical implication between positive topological entropy and distributional chaos of the second type has been shown by Downarowicz, 2014, \cite{Downar}. Note that, in both cases, the converse implications are not true (see \cite{Smit} and \cite{FPS}, respectively). Hence a natural question arises, whether there is a property connected to positiveness of topological entropy is equivalent to the occurrence of Li-Yorke chaos. It has been solved in \cite{FS} by Franzov\'{a} and Sm\'{i}tal for maps on the compact interval.

\begin{theorem}
\label{FSthm}
A dynamical system $(I,f)$ where $I$ is a compact interval and $f$ a continuous mapping from $I$ to itself, is Li-Yorke chaotic if and only if there exists a sequence $S$ such that the topological sequence entropy $h_S(f) > 0$.
\end{theorem}



The concept of a nonautonomous dynamical system (NDS) is a natural generalization of the well known autonomous case. The trajectory of a point in NDS is given by successive application of maps from a given sequence (instead of a single map in the autonomous case). It can be used to introduce e.g. seasonality to the dynamical system or other type of perpetual changes (see e.g., \cite{Dengue, China}). The first paper in the field of nonautonomous discrete dynamical systems is due to Kolyada and Snoha, \cite{KS}. In the paper, the authors introduce the notion of topological entropy for nonautonomous systems using open covers (analogy to the definition for autonomous systems, given in \cite{TE1}) and they show that it is, similarly to the autonomous case, equivalent to the definition using $(n,\varepsilon)$-separated sets (for autonomous case, see \cite{TE2}). They also prove many properties of topological entropy of nonautonomous dynamical systems, mainly convergent and equicontinuous ones. Among all the properties proven in \cite{KS}, let us state at least the commutativity of topological entropy $h(f\circ g) = h(g\circ f)$.

The notion of Li-Yorke chaos for nonautonomous systems was studied, e.g. by \v{S}tef\'{a}nkov\'{a}, \cite{Stef}, who proved that a uniformly convergent nonautonomous system inherits Li-Yorke chaos from its limit system, or by C\'{a}novas in \cite{Cano1}.

The focus of this paper is to prove or disprove a possible generalization of Theorem \ref{FSthm} for the case of uniformly convergent non-autonmous dynamical systems on the interval. It was shown that by a simple modification of an example from the proof of Theorem 1 by Dvo\v{r}\'{a}kov\'{a}, \cite{Dvor}, we obtain a counterexample to one of the implications (see \cite{Sot}). The construction of the second counterexample is the main result of this paper.

\section{Preliminaries}

Let us now sum up some necessary definitions and notions. Unless stated otherwise by $X$ we mean a compact metric space equipped with the metric $\varrho$ and by $I$ the unit compact interval $[0,1]$. By $C(X)$ we mean a collection of all continuous maps from $X$ to itself and by $f$ we mean a map from this collection.

By $f_{1,\infty}$ we denote a sequence of surjective mappings from $C(X)$ (or $C(I)$). Unless stated otherwise we will assume that the sequence $f_{1,\infty}$ converges uniformly to a limit mapping $f$. The surjectivity assumption is necessary since a single constant map may ``destroy'' the dynamics of the NDS, e.g., if $f_1 = 0$ and, for any $i\geq2$, $f_i$ are the tent maps (i.e, $f\!\!: I\rightarrow I, f(x) = 1-|1-2x|$), then there is a single constant trajectory, while if we drop the first map we get Li-Yorke chaotic (autonomous) dynamical system. Moreover, with the surjectivity assumption we have the following implication  for the following implication $h(f) > 0 \Rightarrow h(f_{1,\infty}) > 0$, \cite{Cano1}, and without it we get $h(f_{1,\infty}) \geq h(f)$, \cite{KS}.

For $f\in C(X)$ and $x\in X$ we denote by $f^n(x)$ the $n$-th iteration of $x$ under $f$. We generalize this notion to NDS in the following way:
\begin{eqnarray*}
&& f^0_i(x) := x\\
&& f^1_i(x) := f_i(x)\\
&& f^n_i(x) := (f_{i+n-1} \circ f^{n-1}_i)(x)
\end{eqnarray*}
Note that we usually use $f^n_1(x)$, but sometimes it is useful to start with the $i$th function from $f_{1,\infty}$, instead.

The pair $(X,f)$ is called {\sl autonomous discrete dynamical system} (ADS for short) and the pair $(X,f_{1,\infty})$ is called {\sl nonautonomous discrete dynamical system} (NDS). We recall that $X$ is a compact metric space and $f$ and all members of $f_{1,\infty}$ belong to $C(X)$. The set of all iterations of a point is called {\sl forward orbit} 
By a {\sl trajectory} we mean the sequence $\{f^n(x)\}_{n=1}^\infty$ and similarly for NDS. By an {\sl $\omega$-limit set of a point $x$} we mean the set of the limit points of the orbit of the point. These definitions can be easily extended, e.g., from the orbit of a point to the orbit of a set etc.

Let us now define a generalization of the notion of Li-Yorke chaos for NDS.

\begin{defini}
Let $(X,f_{1,\infty})$ be an NDS and $\varrho$ the corresponding metric on $X$. Two distinct points $x,\ y \in X$ form a {\sl Li-Yorke pair} (shortly {\sl LY-pair}) if and only if it holds:
\begin{eqnarray*}
& \limsup_{n \rightarrow \infty} \varrho(f_1^n(x),f_1^n(y)) > 0\\
& \liminf_{n \rightarrow \infty} \varrho(f_1^n(x),f_1^n(y)) = 0
\end{eqnarray*}
A set where arbitrary two distinct points form a LY-pair is called {\sl Li-Yorke scrambled set}.
A NDS $(X,f_{1,\infty})$ is called {\sl Li-Yorke chaotic} (LYC), if $X$ contains an uncountable Li-Yorke scrambled set.
\end{defini}

Obviously, if $f_n = f$ for every $n\in\mathbb{N}$, we obtain the original definition of LYC for ADS $(X,f)$.

The notion of topological entropy was introduced by Adler, Konheim and MacAndrew, \cite{TE1}, and later equivalently redefined by Bowen, \cite{TE2}. The topological sequence entropy (for ADS) was introduced by Goodman, \cite{Good}. We will state here the definition of topological sequence entropy for NDS in Bowen form since it is more appropriate for our purposes and it was first stated by Kolyada and Snoha in \cite{KS}.

\begin{defini}
Let $X$ be a compact metric space with metric $\varrho$ 
and $f_{1,\infty}$ a~sequence of maps from  $C(X)$. 
Let $A = \{a_i\}_{i=1}^\infty$ be an~increasing sequence of positive integers. For each positive integer $n$ we put $$\varrho_n^A(x,y) := \max_{0 \leq j \leq n-1} \varrho(f^{a_j}_1(x), f^{a_j}_1(y)).$$ A subset $E$ of the space $X$ is then called {\sl $(n,\varepsilon,A)$-separated} if for any two distinct points $x$ and $y$ from $E$ it holds $\varrho_n^A(x,y) > \varepsilon$. We denote the maximal cardinality of an $(n,\varepsilon,A)$-separated set in $X$ by $s_n^A(f_{1,\infty},\varepsilon)$. By a {\sl topological sequence entropy of $f_{1,\infty}$ with respect to the sequence $A$} we mean $$h_A(f_{1,\infty}) := \lim_{\varepsilon\rightarrow 0}\limsup_{n\rightarrow\infty} \frac{1}{a_n}\cdot\log s_n^A(f_{1,\infty},\varepsilon)$$
\end{defini}

If we take, in the above definition, $a_i = i$, we get the notion of  {\sl topological entropy} of a NDS. Obviously, for a constant sequence of maps we get the notions of topological sequence entropy and topological entropy of an ADS.

Since for ADS on interval Theorem \ref{FSthm} holds, the following natural question arises.
\begin{conj}
\label{FConj}
Is it true that a NDS $(I,f_{1,\infty})$ is Li-Yorke chaotic if and only if there exists a sequence $S$ such that the topological sequence entropy $h_S(f_{1,\infty}) > 0$?
\end{conj}



As it was stated in the Introduction a counterexample showing that there exists a NDS $(I,f_{1,\infty})$ such that it is LYC and for any increasing sequences $S$ of integers $h_S(f_{1,\infty}) = 0$, can be found in \cite{Sot}.

The main result of the paper shows that even the other implication does not hold:



\begin{theorem}
\label{Main1}
There exists a NDS $f_{1,\infty}$ of surjective continuous maps from $I$ to itself, such that
\begin{enumerate}
\item it converges uniformly to $f$,
\item there is a sequence $S$ of positive integers such that $h_S(f_{1,\infty}) \geq \log 3$,
\item $(I,f_{1,\infty})$ is not LYC.
\end{enumerate}
\end{theorem}

Let us continue with some notions we will use in the proof. Let $f \in C(I)$ and let $h(f) = 0$. Then for every infinite $\omega$-limit set $\tilde{\omega}$ there exists an {\sl associated system} $\mathcal{I}(\tilde{\omega}) := \{J(k,n)| 0 \leq k < 2^n, n \geq 0\}$ of minimal $f$-periodic compact intervals of period $2^n$ and that $J(n+1,0) \subset J(n,0)$ and
\begin{equation}
\label{AssocSet}
f(M(\tilde{\omega})) = M(\tilde{\omega}) := \bigcap_{n\geq 0} \bigcup_{0\leq k < 2^n} J(n,k) \supseteq \tilde{\omega}.
\end{equation}
The set $M(\tilde{\omega})$ is called {\sl the simple set of the associated system}.

We will also use some notions from symbolic dynamics. By $\Sigma$ we will denote the {\sl shift space} which is a set of all infinite sequences of two symbols, say 0 and 1. This space is equipped with the lexicographical order and the order topology. Let $n_k$ denote a finite sequence of length $k$, $n_k\in\{0,1\}^k$ or a {\sl $k$-block}. Then we denote by $\Sigma_{n_k}$ an {\sl $n_k$-cylinder}, i.e., a subset of $\Sigma$ consisting of those sequences starting with the block $n_k$. The block $n_k$ is in this context called a {\sl code} of the cylinder $\Sigma_{n_k}$.

Let $\underline{n}\in\Sigma$ and $k\in\mathbb{N}$, then by $\underline{n}|k$ we mean a $k$-block consisting of $k$ first symbols of $\underline{n}$. We also use $\underline{0}$ and $\underline{1}$ as a shorthand for constant sequences consisting of 0's or 1's, respectivelly. And by $n_k\underline{0}$ we mean a concatenation of the $k$-block and the infinite sequence \underline{0}.

An important role in our proof plays the adding machine $\alpha\!\!: \Sigma\rightarrow\Sigma$. This mapping adds a sequence $1\underline{0}$ to a mapped sequence modulo 2. E.g. a (piece of) the full trajectory of $\underline{0}$ looks like this:
\begin{equation}
\label{AddMachEx}
\ldots \mapsto 00\underline{1} \mapsto 10\underline{1} \mapsto 0\underline{1} \mapsto \underline{1} \mapsto \underline{0} \mapsto 1\underline{0} \mapsto 01\underline{0} \mapsto 11\underline{0} \mapsto \ldots
\end{equation}
A mapping $\varphi\!:\Sigma \rightarrow \Sigma$ is called a {\sl simple map} if for any $\underline{n} \in \Sigma$ and any $n\in\mathbb{N}$ the cylinder $\Sigma_{\underline{n}|k}$ is $2^k$ periodic. It was proved in \cite{BS} that a map is simple if and only if it is conjugate to the adding machine.

Finally, notice that we will also use the word {\sl block} for a finite sequence of maps. We believe it will always be obvious from the context whether we mean a block of 0's and 1's or of mappings.

\section{Main result}

Before getting to the proof of the main theorem itself, let us first construct an easy example of a NDS which does not converge uniformly, it is made of surjective interval maps, it has positive topological entropy and it has no Li-Yorke pairs. This example is given here for a methodical reason -- we use the ideas of this construction in the proof of the main theorem.

\begin{lemma}
\label{NonUnif}
There exists a NDS $(I,f_{1,\infty})$ where all the maps in $f_{1,\infty}$ are surjective (the sequence does not converge uniformly), it has positive topological entropy and it is not LYC.
\end{lemma}

\begin{proof}
Let $A := \{a_n\}_{n=1}^\infty$ with the first member $a_1 = 1/3$, be a strictly decreasing sequence of reals and such that $\lim_{n \rightarrow \infty} a_n = 0$. Let $B = \{b_n\}_{n=1}^\infty$ with $b_n:= 1-a_n$ and let $K_n := [a_n, b_n]$. This sequence of intervals is strictly increasing in the inclusion-wise sense and its limit is the unit interval $I$. Next, denote the intervals $U_n := [a_{n+1}, a_n]$ and $V_n := [b_n, b_{n+1}]$.

Now, we will use these intervals to construct two sequences of surjective mappings from $C(I)$, out of which we will construct the desired NDS. The first of them is the sequence $\varphi_i$ where
$$\varphi_1(x) = \left\{\begin{array}{rl} x, & x\in[0,\frac13)\cup[\frac23,1]\\ 3x-\frac23, & x\in[\frac13,\frac49]\\ -3x+2, & x\in [\frac49,\frac59)\\ 3x-\frac53, & x\in[\frac59,\frac23) \end{array} \right.$$
and the $\varphi_n$ for $n > 1$ is then defined as identity outside of $K_n$ and a three lap piecewise linear mapping inside it such that $\varphi|_{U_{n-1}}$ as well as $\varphi|_{V_{n-1}}$ are increasing and onto $K_n$ and $\varphi|_{K_{n-1}}$ is decreasing and onto $K_n$ (see Figure \ref{phin}).

\begin{figure}
\begin{subfigure}{.495\textwidth}
\centering
\includegraphics[scale=0.35]{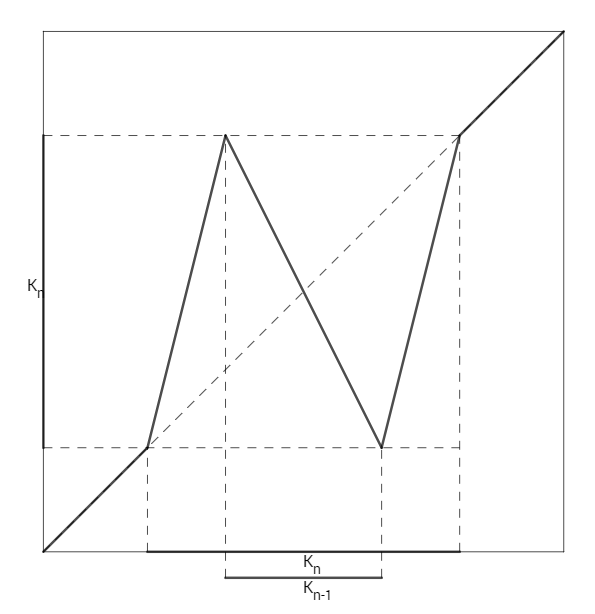}
\caption{\sl The map $\varphi_n$ for $n>1$}
\label{phin}
\end{subfigure}
\begin{subfigure}{.495\textwidth}
\includegraphics[scale=0.35]{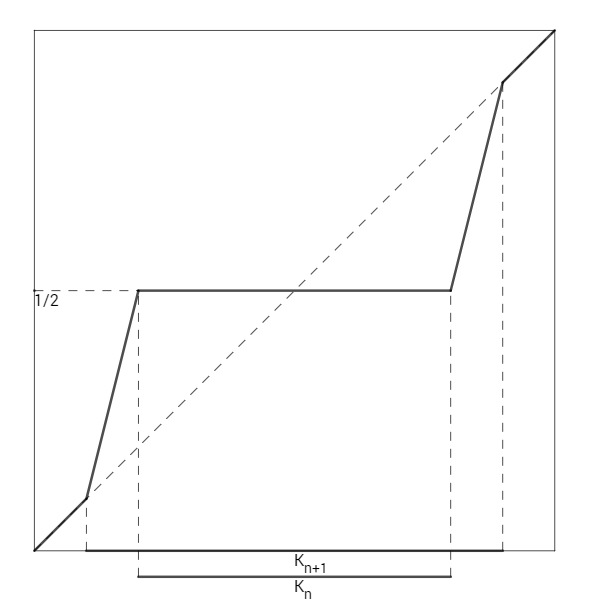}
\caption{\sl The map $\psi_n$}
\label{psin}
\end{subfigure}
\caption{\sl Sketches of maps from proof of Lemma \ref{NonUnif}}
\end{figure}

Let us proceed to the second sequence, $\psi_n$, it is again identity but in this case outside of $K_{n+1}$, constant and equal to 1/2 on $K_n$ and linearly continuously extended in between (see Figure \ref{psin}).

The desired NDS is constructed by blocks of $i_k$ repetitions of $\varphi_k$ followed by a single map $\psi_k$. Particularly $$f_{1,\infty} := B_1, B_2, \ldots$$ where for $k\in\mathbb{N}$ $$ B_k := \underbrace{\varphi_k, \ldots, \varphi_k}_{i_k - \mbox{\scriptsize times}}, \psi_k.$$
If we consider the finite sequence (or block) $B_k$ as a composition of its members, we get $$B_k(x) := \psi_k \circ \varphi_k \circ \ldots \circ \varphi_k (x) = \psi_k(x)$$
and hence $B_k \circ B_{k-1} \circ \ldots \circ B_1 (x) = \psi_k(x)$. So, the trajectory of every point is eventually constant (with the value 1/2). That excludes any possibility for having a LY-pair, and hence such NDS is not LYC.

Finally, let us prove that the topological entropy is positive. Before collapsing by the last map $\psi_k$ every block creates a 3-horseshoe (i.e., there are nonempty intervals, pairwise intersecting at most at one point, $U$, $V$ and $W$ such that $i_k$-th iterate under $\varphi_k$ of any of them covers the union $U\cup V \cup W$) and hence there is an $(i_k,\varepsilon)$-separated set of cardinality $s_k \geq 3^{i_k}$ in the interval $K_k$ for an $\varepsilon < 1/9$. And consecutively, $f_{1,\infty}$ has an $(m_k,\varepsilon)$-separated set of cardinality $s_k \geq 3^{i_k}$ where $m_k := i_1 + i_2 + \ldots + i_k + k$. Then, thanks to the surjectivity of $\varphi$'s, if the sequence $i_k$ increases rapidly enough we get $$ h(f_{1,\infty}) \geq \limsup_{k \rightarrow \infty} \frac{\log s_k}{m_k} = \limsup_{k \rightarrow \infty} \frac{\log s_k}{i_k} = \log 3 > 0$$
\end{proof}

The counterexample from the Theorem \ref{Main1} will be constructed using the blowing-up orbit technique, it is based on result by Bruckner and Sm\'{i}tal \cite{BS}. In the following theorem we construct the uniform limit of our wanted NDS using the blowing-up orbits technique and symbolic dynamics and later on we construct the counterexample from the Theorem \ref{Main1} itself.

\begin{theorem}
\label{ThmLim}
There is a surjective map $f$ from $C(I)$ such that:
\begin{enumerate}
\item $f$ has exactly one infinite $\omega$-limit set $\tilde{\omega}$ and it has zero topological entropy;
\item $f$ is not LYC;
\item the simple set $M(\tilde{\omega})$ generated by the associate system $\mathcal{J}(\tilde{\omega})$ has non-empty interior;
\item for the system $\{G_n\}_{n\in\mathbb{Z}}$ of non-degenerate components of $M(\tilde{\omega})$ it holds: $\forall n\in\mathbb{Z}\ f(G_n) = G_{n+1}$.
\end{enumerate}
\end{theorem}
Existence of such mapping is known. A general method for such construction can be found for example in \cite{BS}. But let us delve into the proof as we will need some parts of it in later constructions.

\begin{proof}
Let us denote the Cantor middle-third set by $Q \subset [0,1]$ and let us identify it with $\Sigma$ via an increasing homeomorphism $\theta: \Sigma \rightarrow Q$. Let us denote by $B_0$ a set of those points in $\Sigma$ that end with $\underline{0}$, by $B_1$ those that end with $\underline{1}$ and by $B$ their union. Then $\bar{\alpha} := \theta \circ \alpha \circ \theta^{-1}$, where $\alpha\!\!: \Sigma \rightarrow \Sigma$ is the adding machine as defined in Section 2 -- see (\ref{AddMachEx}), is a simple map from $Q$ to itself and $\bar{B} = \theta(B)$ is its full orbit. And $\bar{B} = \{b_i\}_{i = -\infty}^\infty$, $b_0 = \theta(\underline{0})$ is also exactly the orbit which we will blow up.

Let $\{G_j\}_{j\in\mathbb{Z}}$ be a collection of disjoint compact intervals in $(0,1)$ indexed in such manner that $G_i$ is to the left of $G_j$ if and only if $b_i < b_j$. For this one-to-one correspondence between $b$'s and $G$'s we refer to $b_j$ (or rather $\theta^{-1}(b_j)$) as to a {\sl code} of $G_j$. Let $h\!: Q \rightarrow 2^{(0,1)}$ be a strictly increasing set-valued map such that $h(x)$ is a point (or a singleton) for every $x\not\in B$ and $h(b_j) = G_j$. Such map is unique since $\bar{B}$ is dense in $Q$. Let us denote $\bigcup h(Q) =: R$. We can look at the inverse $h^{-1}$ as on a non-decreasing continuous map $R \rightarrow Q$ which is constant on every $G_j$. Let $f^\prime\!: R \rightarrow R$ be a continuous map semiconjugate to the adding machine $\bar{\alpha}$, i.e., $h^{-1}\circ f^{\prime} = \bar{\alpha}\circ h^{-1}$. Let $G^\prime_j$ denote the interval $G_j$ without its right endpoint if $b_j\in B_0$ or without its left endpoint if $b_j\in B_1$, respectivelly. Then, $\tilde{\omega} := R\setminus \bigcup_j G^\prime_j$ is a Cantor set and it is minimal for the map $f^\prime$. It is quite clear that the simple set $M(\tilde{\omega}) = R$ (due to density of $B$ in $\Sigma$ and density of $G^\prime_j$ in $R$), see (\ref{AssocSet}).

Finally, we extend $f^\prime$ linearly to $f$ -- a continuous map on the whole $I$ -- so that $f$ is constant on the intervals $[0,\min R]$ and $[\max R, 1]$. The set $\tilde{\omega}$ is obviously the unique infinite $\omega$-limit set of $f$ and the topological entropy of $f$ cannot exceed the topological entropy of $f^\prime$ which is zero. So far we have proved the first two demanded properties. The other two follow from the construction.
\end{proof}

\begin{figure}
\begin{subfigure}{0.495\textwidth}
\centering
\includegraphics[scale=.35]{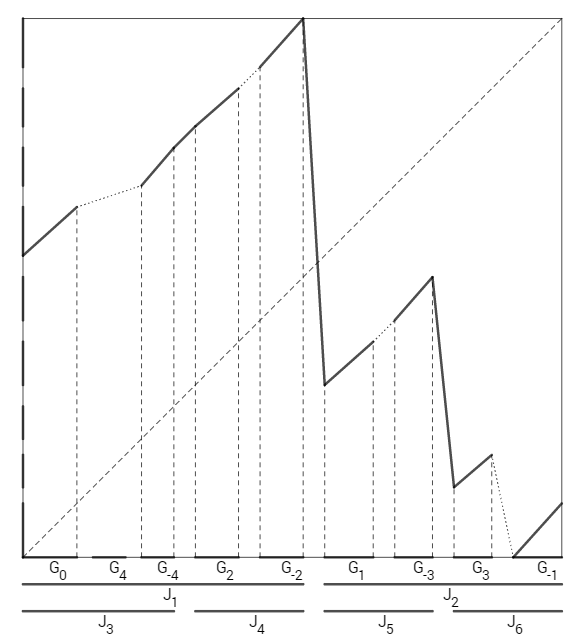}
\caption{\sl The map $f$}
\end{subfigure}
\begin{subfigure}{0.495\textwidth}
\centering
\includegraphics[scale=.35]{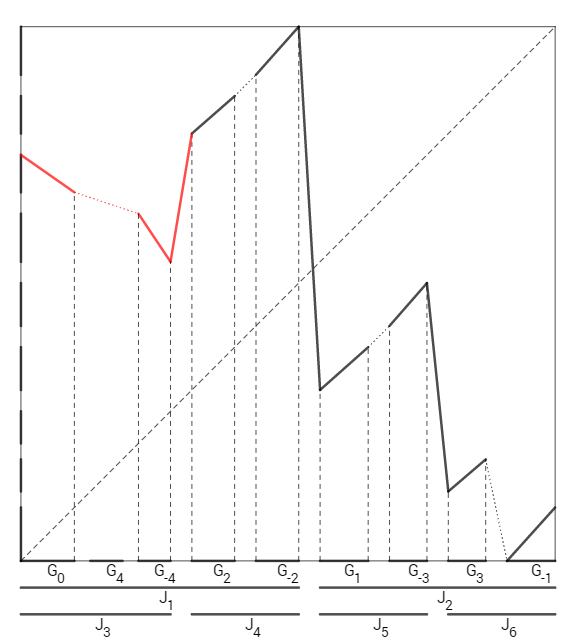}
\caption{\sl The map $\eta_{00}$}
\end{subfigure}
\caption{\sl Sketches of limit map $f$ and its perturbation, dotted lines are unspecified parts of the graph}
\end{figure}

Before we proceed we will need a lemma which uses the following map: $\tau_{n_k}\!: \Sigma \rightarrow \Sigma$ is defined for a block $n_k \in \{0,1\}^k$ to be the identity if the first $k$ symbols of a sequence $\underline{n}\in\Sigma$ are distinct from the block $n_k$ and a 0-1-after-$k$-symbols-reversing map otherwise. To make things clearer, the map $\tau_{n_k}$ does not change the first $k$ symbols of a $n_k\underline{n}$ and swaps the zeros to ones and conversely ones to zeros at the very same time in the rest of the sequence, e.g. $$\tau_{n_k}(n_k110100\ldots) = n_k001011\ldots$$ Such map is continuous on $\Sigma$ regardless of $n_k$.

\begin{lemma}\label{lm}
Let $n_k\in\{0,1\}^k$ and let $\eta\!: \Sigma \rightarrow \Sigma$ denote the map $\alpha\circ\tau_{n_k}$. Then
\begin{enumerate}
\item $\eta$ is continuous,
\item $\underline{0}$ is a $2^k$-periodic point of $\eta$ and it interesects the cylinder $\Sigma_{n_k}$ at a single point $n_k\underline{0}$,
\item every point $\underline{n} \in B$ is either a $2^k$-periodic point of $\eta$ and belongs to an orbit of $\underline{0}$, or it is an $m\cdot2^k$-periodic point of $\eta$ for some (positive) integer $m$.
\end{enumerate}
\end{lemma}

\begin{proof}
Continuity is obvious as $\eta$ is a composition of two continuous maps.

Even the second part of this lemma can be proven quite easily by a straightforward calculation.
\begin{equation} \label{etas} \eta^{2^k}(\underline{0}) = \eta^{2^k - e(n_k)}(n_k\underline{0}) = \eta^{2^k - e(n_k) - 1}(\alpha(n_k)\underline{1}) = \eta(\underline{1}) = \underline{0}, \end{equation}
where $e$ is so called evaluation map which assigns a sum $e(x) = \sum_{i=1}^k 2^{i-1} x_k$ to a $k$-block $x = x_1 x_2 x_3\dots x_k $.

In the first equality of (\ref{etas}) we just kept adding up $1\underline{0}$ obviously moving each time to a different cylinder (defined by a $k$-block) until we reached the $\Sigma_{n_k}$. The second equality is the only one where $\tau_{n_k}$ acutally affects things, switches the ending zeros to ones and then we again add $1\underline{0}$ to move to yet another cylinder. In the second equality we may assume that $n_k$ contains at least one zero, otherwise we could just skip right to the end. In the third equality we again just add ones from the beginning contiuing our tour through the distinct $k$-block cylinders and the fourth one is all clear.

As a by-product of the proof of the second part we can see that $n_k\underline{0}$ (and all the other points of the orbit) is $2^k$-periodic as well.

The third part of the proof follows from the facts that $\eta$ is basically adding machine that in one cylinder switches from followers of $\underline{0}$ to its predecessors and that running through all cylinders takes $2^k$ iterations.
\end{proof}

Now we are ready to prove our main result, i.e., Theorem \ref{Main1}. As we will construct it as a sequence of perturbations of a limit mapping from Theorem \ref{ThmLim}, we will restate it with some details added as a Theorem \ref{Main2}.

\begin{theorem}
\label{Main2}
Let $f$ be a surjective continuous map from $I$ to itself, satisfying conditions (i)-(iv) from the Theorem \ref{ThmLim}. Then there is a nonautonomous system $f_{1,\infty}$ of surjective continuous maps from $I$ to itself, such that
\begin{enumerate}
\item it converges uniformly to $f$,
\item there is a sequence $S$ of positive integers such that $h_S(f_{1,\infty}) \geq \log(3)$,
\item $(I,f_{1,\infty})$ is not LYC.
\end{enumerate}
\end{theorem}

\begin{proof}
Let $G_j$ be the sets defined in the proof of Theorem \ref{ThmLim}, let $|G|$ denote the length of the interval $G$ and put $\varepsilon_0 := |G_0|/3$. Without loss of generality we can assume that $G_0$ is located to the left of any other $G_j$ (so its code is $\underline{0}$). Since the intervals $G_j$ are pairwise disjoint and $f(G_j) = G_{j+1}$, we may assume that $f|_{G_j}$ is a linear and strictly monotone map onto $G_{j+1}$.

The construction of the nonautonomous system is, similarly to the construction in the proof of the Theorem \ref{NonUnif}, constructed by blocks. We start with an increasing sequence $K^1_0 \subset K^2_0 \subset K^3_0 \ldots \subset G_0$ of compact intervals. The center of every $K^j_0$ is the center $c_0$ of the interval $G_0$, the length $|K^0_1| > \varepsilon_0$ and $\lim_{n\rightarrow\infty} K^n_0 = G_0$. For any $j\in\mathbb{Z}$ and any $n\in\mathbb{N}$, let $K^n_j := f^j(K^n_0)$. Because $f$ is a linear map from $G_j$ onto $G_{j+1}$ we have again for any $j\in\mathbb{Z}$ and any $n\in\mathbb{N}$
\begin{multline*}
K^1_j \subset K^2_j \subset K^3_j \subset \ldots \subset G_j, \lim_{n\rightarrow\infty} K^j_n = G_j\\ \mbox{and } f\!:\ldots \mapsto K^n_{-2} \mapsto K^n_{-1} \mapsto K^n_0 \mapsto K^n_1 \mapsto \ldots
\end{multline*}

By Theorem \ref{ThmLim}, there are compact periodic intervals $\{J_i\}_{i\geq 1} \subset \mathcal{I}(\tilde{\omega})$ with codes $n_i$ such that the period of $J_i$ is $2^{k_i}$ (cf. Lemma \ref{lm}) and $\lim_{i\rightarrow\infty} |J_i| = 0$.

Let us define $\lambda_{n_i}: I\rightarrow I$ as an analog to the $\tau$ from Lemma \ref{lm}. On the interval $J_i$ it maps the underlying interval $G_j$ (linearly) onto another $G_l$ if and only if $\tau_{n_i}$ maps the respective codes one to another as well and it is linearly extended to the rest of the $J_i$. On the other intervals from the periodic orbit of $J_i$ it is defined as identity and then linearly extended to the whole $I$.

The interval $J_i$ contains intervals from the orbit of $G_0$ with codes $n_i\underline{1}$ (a predecessor of $G_0$) and $n_i\underline{0}$ (a follower of $G_0$), and $\lambda_{n_i}$ maps the follower onto the predecessor and so completes the periodic cycle. More precisely, let $\eta_i := f \circ \tau_{n_i}$. Then, for any $i \geq 1$, we have:
\begin{enumerate}
\item $\lim_{i\rightarrow\infty}||f-\eta_i|| = 0$
\item Intervals $G_j$ with codes $n_i\underline{0}$ and all corresponding subintervals $K^n_j$ are $\eta_i$-periodic with period $2^{k_i}$.
\item $\eta_i$-periodic orbit of $G_0$ and of any $K^n_0$ intersects $J_i$ at a single interval $G_{p_i}$ or $K^n_{p_i}$, resepectivelly, where $p_i = e(n_i)$. In particular, code of $G_{p_i}$ is $n_{k_i}\underline{0}$.
\item Every interval $G_j$ from the $f$-orbit of $G_0$ is $\eta_i$-periodic. Its period is either $2^{k_i}$ and belongs to an $\eta_i$-orbit of $G_0$, or it is $m\cdot 2^{k_i}$ for some integer $m>1$.
\end{enumerate}

All the above facts except the first one are direct consequences of the definition of $\eta$ and $\lambda$ and Lemma \ref{lm}. The first one is the consequence of the fact that the supremum of distances of $f$ and $\eta_i$ is the length of $f$-image (or $\eta_i$-image, the lengths are the same) of $J_i$, which is positive and converges to zero.

As a next step, we define the maps analogical to the ones from the proof of Theorem \ref{NonUnif}. Let $\varphi_{i,n}(x)=f(x)$ for $x\not\in\mbox{int}(K^n_{p_i})$ and let it be a three lap piecewise monotone map on the interior. In the case of $n=0$ we impose no further restrictions, in all the other cases it maps each of the three parts of $K^n_{p_i}$ obtained by dividing it by $K^{n-1}_{p_i}$ back onto the whole $K^n_{p_i}$. With these maps we can define the following sequence $g_{1,\infty}$ for some $i$ and $n$: $$g_{1,\infty} = \varphi_{i,n}\circ\lambda_i,\underbrace{\eta_i,\eta_i,\ldots,\eta_i}_{2^{k_i}-1\ \mbox{times}},\varphi_{i,n}\circ\lambda_i,\underbrace{\eta_i,\eta_i,\ldots,\eta_i}_{2^{k_i}-1\ \mbox{times}},\varphi_{i,n}\circ\lambda_i,\underbrace{\eta_i,\eta_i,\ldots,\eta_i}_{2^{k_i}-1\ \mbox{times}},\ldots$$
The $m$th iterate $g_1^m$ with $m < 2^{k_i}$ has 3 laps mapping $K^n_{p_i}$ onto $K^n_{p_i+m}$ and the others ($j\neq p_i$) $K^n_j$ linearly onto $K^n_{j+m}$ (so the perturbed part travels along the orbit of $K^n_{p_i}$). Similarly, for $2^{k_i} \leq m < 2\cdot2^{k_i}$, the $m$th iterate has $3^2$ laps on $K^n_{p_i}$, $3^3$ laps for $2\cdot2^{k_i} \leq m < 3\cdot2^{k_i}$ etc.

Let $q_i := 2^{k_i}-p_i$, i.e., $f^{q_i}(G_{p_i}) = G_0$ and let $R = \{r_m\}_m\in\mathbb{N}$ be a sequence defined as $r_m := q_i - 1 + m\cdot2^{k_i}$. The sequence represents the times when the perturbed part of the $g_{1,\infty}$ has values in $K^n_0$. Since $|K^1_0| > \varepsilon_0$ the laps have height greater than $\varepsilon_0$. Consequently, the topological sequence entropy $h_R(g_{1,\infty}) \geq \log 3$, since there is an $(\varepsilon_0,r_m)$-separated set of cardinality $3^k$

Let $\psi_{i,n} = f$ outside of the interior of $K_{p_i}^{n+1}$, let $\psi_{i,n}|_{K_{p_i}^n}$ be constant with value $c_{p_i+1}$ (the center of $G_{p_i}$) and let us extend it continuously and linearly on $K_{p_i}^{n+1}\setminus K_{p_i}^{n}$. The role of $\psi_{i,n}$ is to ``kill'' LY-scrambled sets. With $\psi_{i,n}$ we finally define the desired NDS. Let $\{a_n\}_{n=1}^\infty$ be an increasing sequence of positive integers and let $b_n := a_n\cdot 2^k_n +1$. Then we define $f_{1,\infty}$ as a sequence of blocks $B_n$ of length $b_n$ where: $$B_n := \varphi_{n,n}\circ\lambda_n, \underbrace{\eta_n,\eta_n, \ldots, \eta_n}_{2^{k_n}-1\ \mbox{times}}, \ldots, \varphi_{n,n}\circ\lambda_n, \underbrace{\eta_n,\eta_n, \ldots, \eta_n}_{2^{k_n}-1\ \mbox{times}}, \psi_{n,n}.$$
Since for any $n$ the trajectory of every $x\in K_{p_i}^n$ is eventually constant in the NDS $B_n,B_n,B_n\ldots$ we can clearly see that every $f_{1,\infty}$-trajectory of a point in any $G_j$ is eventually constant as well. As a consequence we can see that $f_{1\infty}$ has no LY-pairs as $\bigcup G_j$ is the only subset of $I$ where LY-pairs can appear because $f_{1,\infty}$ outside of the $\bigcup G_j$ is conjugate to the adding machine which means it is distal (i.e., infimum of distances of all iterations of two distinct points is positive).

Next, take a sequence $S := \{s_n\}_{n=1}^\infty$ where $$s_1,s_2,\ldots, s_{a_1} = q_1 - 1 + 2^{k_1}, q_1 - 1 + 2\cdot 2^{k_1}, \ldots, q_1 - 1 + a_1\cdot 2^{k_1}$$
and, taking $\tilde{a}_0 := 0$ and $\tilde{a}_k := a_1 + a_2 + \ldots + a_k$, $$s_{\tilde{a}_{n-1}+k} = s_{\tilde{a}_{n-1}} + q_n - 1 +k\cdot 2^{k_n}, \mbox{ where } 1 \leq k \leq \tilde{a}_{n-1}$$
If we assume that the sequence $\{a_n\}$ increases rapidly enough, then $h_S(f_{1,\infty}) \geq \log 3$. To look at this statement a bit closely, we can say that $f_{1,\infty}$ generates an $(\varepsilon_0,s_{\tilde{a}_n}))$-separated set of cardinality greater or equal to $3^{a_n-1}$ whence
$$h_S(f_{1,\infty}) \geq \limsup_{n\rightarrow\infty} \frac{a_n-1}{\tilde{a}_n}\log 3 = \limsup_{n\rightarrow\infty} \frac{a_n-1}{a_n}\log 3 = \log 3.$$

The uniform convergence is a consequence of the fact, that $\lim_{i \rightarrow \infty}|J_i| = 0$ and that all the maps $\lambda_i$, $\varphi_{i,n}$ and $\psi_{i,n}$ are different from $f$ only at the interval $J_i$.

\end{proof}

\section{Acknowledgement}
The author gratefully acknowledges valuable discussions, remarks and thorough proofreading by his supervisor Marta \v{S}tef\'{a}nkov\'{a}. The author is also thankful for the support of the grant SGS/18/2016.


\begin{thebibliography}{19}
\bibitem{TE1}
	R. L. Adler, A. G. Konheim, M. H. McAndrew,
	Topological entropy,
	Trans. Amer. Math. Soc. 114 (1965),
	309--319
\bibitem{BGKM}
	F. Blanchard, E. Glasner, S. Kolyada and A. Maass,
	On Li-Yorke pairs,
	J. Reine Angew. Math. 547 (2002),
	51--68
\bibitem{TE2}
	R. Bowen,
	Periodic points and measures for axiom A diffeomorphims,
	Trans. Amer. Math. Soc. 154 (1971),
	125--136
\bibitem{BS}
	A. M. Bruckner, J. Sm\'{i}tal,
	A characterization of $\omega$-limit sets of maps of the interval with zero topological entropy,
	Ergod. Theory and Dyn. Syst. 13 (1993),
	7--19
\bibitem{Dengue}
	M. N. Burattini, F. A. B. Coutinho, L. F. Lopez, E. Massad,
	Threshold conditions for a non-autonomous epidemic system describing the population dynamics of dengue,
	Bull. Math. Biol. 68 (2006),
	2263--2282
\bibitem{Cano1}
	J. S. C\'{a}novas,
	Li-Yorke chaos in a class of nonautonmous discrete systems,
	J. Difference Equ. Appl. 17 (2011),
	479--486
\bibitem{Downar}
	T. Downarowicz,
	Positive topological entropy implies chaos DC2,
	Proc. Amer. Math. Soc. 142 (2014),
	137--149
\bibitem{Dvor}
	J. Dvo\v{r}\'{a}kov\'{a},
	Chaos in nonautonomous discrete dynamical systems,
	Commun. Nonlinear. Sci. Numer. Simulat. 17 (2012),
	4649--4652
\bibitem{FPS}
	G. L. Forti, L. Paganoni, J. Sm\'{i}tal,
	Dynamics of homeomorphisms on minimal sets generated by triangular mappings,
	Bull. Austral. Math. Soc. 59 (1999),
	1--20
\bibitem{FS}
	N. Franzov\'{a}, J. Sm\'{i}tal,
	Positive sequence topological entropy characterizes chaotic maps,
	Proc. Amer. Math. Soc. 112 (1991),
	1083-1086
\bibitem{China}
	S. Gao, Y. Liu, Y. Zhang,
	Analysis of a nonautonomous model for migratory birds with saturation incidence rate,
	Commun. Nonlinear. Sci. Numer. Simul. 17 (2012),
	1659--1672
\bibitem{Good}
	T. N. T. Goodman,
	Topological sequence entropy,
	Proc. London Math. Soc. 29 (1974),
	331--350
\bibitem{KS}
	S. Kolyada, \v{L}. Snoha,
	Topological entropy of nonautonomous dynamical systems,
	Random \& Computational Dynamics 4 (1996),
	205--233
\bibitem{Smit}
	J. Sm\'{i}tal,
	Chaotic functions with zero topological entropy,
	Trans. Amer. Math. Soc. 297 (1986),
	269--282
\bibitem{Sot}
	J. \v{S}otolov\'{a},
	Topological sequence entropy of nonautonomous dynamical systems,
	Silesian University in Opava,
	Diploma thesis (2016)
\bibitem{Stef}
	M. \v{S}tef\'{a}nkov\'{a},
	Inheriting of chaos in uniformly convergent nonautonmous dynamical systems on the interval,
	Discrete Contin. Dyn. Syst. 36 (2016),
	3435--3443
\end{thebibliography}
\end{document}